\documentclass[11pt]{amsart}
\usepackage{amsmath}
\usepackage{amssymb}
\usepackage{amscd}

\def\NZQ{\mathbb}               
\def\NN{{\NZQ N}}
\def\QQ{{\NZQ Q}}
\def\ZZ{{\NZQ Z}}
\def\RR{{\NZQ R}}

\newtheorem{Theorem}{Theorem}[section]

\newtheorem{Corollary}[Theorem]{Corollary}
\newtheorem{Proposition}[Theorem]{Proposition}
\newtheorem{Remark}[Theorem]{Remark}

\let\epsilon\varepsilon
\let\phi=\varphi
\let\kappa=\varkappa

\begin{document}

\title{Asymptotic growth of saturated powers and epsilon multiplicity}
\author{Steven Dale Cutkosky}
\thanks{Partially supported by NSF}

\address{Steven Dale Cutkosky, Department of Mathematics,
University of Missouri, Columbia, MO 65211, USA}
\email{cutkoskys@missouri.edu}


\maketitle
\section{Introduction}

In this paper, we study the growth of saturated powers of
modules. In the case of an ideal $I$ in a local ring $(R,\mathfrak m)$, the saturation of $I^k$ in $R$ is
$$
(I^k)^{\rm sat}=I^k:_{R}\mathfrak m^{\infty}=\cup_{n=1}^{\infty}I^k:_R\mathfrak m^n.
$$
There are examples showing that the algebra of saturated powers of $I$, 
$\bigoplus_{k\ge 0}(I^k)^{\rm sat}$ is not a finitely generated $R$-algebra;
for instance, in many cases the saturated powers are the symbolic powers.
As such, it cannot be expected that the ``Hilbert function'',
giving the length of the $R$-module $(I^k)^{\rm sat}/I^k$, is very well behaved for large $k$.  However, it can be shown that it is bounded above by a polynomial in $k$ of degree $d$, where $d$ is the dimension of $R$.
We show that in many cases, there is a reasonable asymptotic behavior of this length.

Suppose that $(R,\mathfrak m)$ is a Noetherian local domain of dimension $d\ge 1$. Let $L$ be the quotient field of $R$. Let $\lambda(M)$ denote the length of an $R$-module $M$. Let $F$ be a finitely generated free
$R$-module, and let $E$ be a submodule of $F$ of rank $e$. Let $S=R[F]=\mbox{Sym}(F)=\bigoplus_{k\ge 0}F^k$ and let $R[E]=\bigoplus_{k\ge 0}E^k$ be the $R$-subalgebra of $S$ generated by $E$. Let
$$
E^k:_{F^k}\mathfrak m^{\infty}=\cup_{n=1}^{\infty}E^k:_{F^k}\mathfrak m^n
$$
denote the saturation of $E^k$ in $F^k$. We prove the following theorem:

\begin{Theorem}\label{ThmInt}  Suppose that $(R,\mathfrak m)$ is a local domain of depth $\ge 2$ which is essentially of finite type over a field $K$ of characteristic zero (or over a perfect field $K$ such that $R/\mathfrak m$ is algebraic over $K$). Let $d$ be the dimension of $R$. Suppose that $E$ is a rank $e$ submodule of a finitely generated free $R$-module $F$. Then the limit
\begin{equation}\label{eqInt1}
\lim_{k\rightarrow \infty}\frac{\lambda(E^k:_{F^k}\mathfrak m^{\infty}/E^k)}{k^{d+e-1}}\in \RR
\end{equation}
exists.
\end{Theorem}

The conclusions of this theorem follow from Theorem \ref{Theorem1*} and  Remark \ref{charp}.

Theorem \ref{ThmInt} is proven in the case when $E=I$ is a homogeneous ideal and $R$ is a  standard graded normal $K$-algebra in our paper \cite{CHST} with H\`a, Srinivasan and Theodorescu. The theorem is proven with the additional assumptions that $R$ is regular, $E=I$ is an ideal in $F=R$, and the singular locus of $\mbox{Spec}(R/I)$ is $\mathfrak m$ in our paper \cite{CHS} with Herzog and Srinivasan.  Kleiman \cite{K} has proven Theorem \ref{ThmInt} in the case that $E$ is a direct summand of $F$ locally at every nonmaximal prime of $R$. The theorem is proven for $E$ of low analytic deviation in \cite{CHS}, for the case of ideals, and by Ulrich and Validashti \cite{UV}  for the case of modules; in the case of low analytic deviation, the limit is always zero. A generalization of this problem to the case of saturations with respect to non $\mathfrak m$-primary ideals is investigated by Herzog, Puthenpurakal and Verma in \cite{HPV}; they show that an appropriate limit exists for monomial ideals.

An example in \cite{CHST} shows that even in the case when $E$ is an ideal $I$ in a regular local ring $R$, the  limit may be irrational.

An important technique in the proof of Theorem \ref{ThmInt} is to use a theorem of Lazarsfeld \cite{L} showing that the volume of a line bundle on a complex projective variety can be expressed as a limit of numbers of global sections of powers of the line bundle;
Lazarsfeld's theorem is deduced from an approximation theorem of Fujita \cite{F2} (generalizations of Fujita's result to positive characteristic are given in \cite{Ta} and \cite{RM}).

We can interpret our results in terms of local cohomology. Let $F_L^k=F^k\otimes_R L$, where $L$ is the quotient field of $R$, so that we have natural
embeddings $E^k\subset F^k\subset F_L^k$ for all $k$. We have identities
$$
H^0_{\mathfrak m}(F^k/E^k)\cong E^k:_{F^k}\mathfrak m^{\infty}/E^k
\mbox{ and }H^1_{\mathfrak m}(E^k)\cong E^k:_{F_L^k}\mathfrak m^{\infty}/E^k.
$$
Further, these two modules are equal if $R$ has depth $\ge 2$.

We thus obtain the following corollary to Theorem \ref{ThmInt}, which shows that the epsilon multiplicity
$\epsilon(E)$ of a module, defined as a limsup in \cite{UV}, actually exists as a limit.

\begin{Corollary}\label{CorInt} Suppose that $(R,\mathfrak m)$ is a local domain of depth $\ge 2$ which is essentially of finite type over a field $K$ of characteristic zero
(or over a perfect field $K$ such that  $R/\mathfrak m$ is algebraic over $K$). Let $d$ be the dimension of $R$. Suppose that $E$ is a rank $e$ submodule of a finitely generated free $R$-module $F$. Then the limit
$$
\lim_{k\rightarrow \infty}\frac{(d+e-1)!}{k^{d+e-1}}\lambda(H^0_{\mathfrak m}(F^k/E^k))\in \RR
$$
exists. Thus the epsilon multiplicity $\epsilon(E)$ of $E$ exists as a limit.
\end{Corollary}

 By the above identities of
local cohomology, we see that (\ref{eqInt1}) is equivalent to the statement that
\begin{equation}\label{eqInt2}
\lim_{k\rightarrow\infty}\frac{H^0_{\mathfrak m}(F^k/E^k)}{k^{d+e-1}}
=\lim_{k\rightarrow\infty}\frac{H^1_{\mathfrak m}(E^k)}{k^{d+e-1}}\in \RR
\end{equation}
exists when $\mbox{depth}(R)\ge 2$.

In Section \ref{Exts}, we extend our results to domains of dimension $d\ge 2$.
We prove the following extension of Theorem \ref{ThmInt}, which shows that the second limit of (\ref{eqInt2}),
$$
\lim_{k\rightarrow\infty}\frac{H^1_{\mathfrak m}(E^k)}{k^{d+e-1}}\in \RR
$$
exists when $R$ is a domain of dimension $d\ge 2$.

\begin{Theorem}\label{Theorem3**} Suppose that $(R,\mathfrak m)$ is a  local domain of
  dimension $d\ge 2$
 which is essentially of finite type over a field $K$ of characteristic zero
 (or over a perfect field $K$ such that $R/\mathfrak m$ is algebraic over $K$).
Suppose that $E$ is a rank $e$ submodule of a finitely generated free $R$-module $F$.
Then the limit
\begin{equation}\label{**eq1}
\lim_{k\rightarrow \infty}\frac{\lambda\left(E^k:_{F_L^k}\mathfrak m^{\infty}/E^k\right)}{k^{d+e-1}}\in\RR
\end{equation}
exists.
\end{Theorem}
Theorem \ref{Theorem3**} follows from Theorem \ref{Theorem3*} and equations (\ref{eqE}) and (\ref{eqE'}).
We prove that the first limit of (\ref{eqInt2}),
$$
\lim_{k\rightarrow\infty}\frac{H^0_{\mathfrak m}(F^k/E^k)}{k^{d+e-1}}\in\RR
$$
exists when $R$
is a  domain of dimension $d\ge 2$ and $E$ is embedded in $F$ of rank $<d+e$. I thank Craig Huneke, Bernd Ulrich and Javid Validashti for pointing out this interesting consequence of Theorem \ref{Theorem3**}.

\begin{Corollary}\label{Cor4**} Suppose that $(R,\mathfrak m)$ is a  local domain of
  dimension $d\ge 2$
 which is essentially of finite type over a field $K$ of characteristic zero
 (or over a perfect field $K$ such that $R/\mathfrak m$ is algebraic over $K$).
Suppose that $E$ is a rank $e$ submodule of a finitely generated free $R$-module $F$.
Suppose that $\gamma=\mbox{\rm rank}(F)<d+e$.
Then the limits
\begin{equation}\label{**eq2}
\lim_{k\rightarrow \infty}\frac{\lambda\left(E^k:_{F^k}\mathfrak m^{\infty}/E^k\right)}{k^{d+e-1}}\in \RR
\end{equation}
and
\begin{equation}\label{**eq3}
\lim_{k\rightarrow \infty}\frac{(d+e-1)!}{k^{d+e-1}}\lambda(H^0_{\mathfrak m}(F^k/E^k))\in \RR
\end{equation}
exist. In particular,  the epsilon multiplicity $\epsilon(E)$ of $E$ exists as a limit.
\end{Corollary}

In the case when $e=1$ and $F=R$, we get the following statement.

\begin{Corollary} Suppose that $(R,\mathfrak m)$ is a local domain of dimension $d\ge 1$ which is essentially of finite type over a field $K$ of characteristic zero
(or over a perfect field $K$ such that  $R/\mathfrak m$ is algebraic over $K$).  Suppose that $I$ is an ideal in $R$. Let 
$(I^k)^{\rm sat}=I^k:_R\mathfrak m^{\infty}$ be the saturation of $I^k$.
Then the limit
$$
\lim_{k\rightarrow \infty}\frac{\lambda((I^k)^{\rm sat}/I^k)}{k^{d}}\in\RR
$$
exists.
\end{Corollary}

Asymptotic polynomial like behavior of the length of extension functions is
studied by Katz and Theodorescu \cite{KT}, Theodorescu \cite{T} and Crabbe, Katz, Striuli and Theodorescu \cite{CKST}. 
By the local duality theorem, we obtain the following corollary to Theorem \ref{ThmInt}.

\begin{Corollary} Suppose that $(R,\mathfrak m)$ is a Gorenstein local domain of dimension  $d\ge 2$ which is essentially of finite type over a field $K$ of characteristic zero
(or over a perfect field $K$ such that  $R/\mathfrak m$ is algebraic over $K$).  Suppose that $E$ is a rank $e$ submodule of a finitely generated free $R$-module $F$. Then the limit
$$
\lim_{k\rightarrow \infty}\frac{\lambda({\rm Ext}^d_R(F^k/E^k,R))}
{k^{d+e-1}}\in \RR
$$
exists. 
\end{Corollary}

\section {Preliminaries}\label{Prep}

Suppose that $(R,\mathfrak m)$ is a Noetherian local domain of dimension $d\ge 1$ with quotient field $L$.
Let $\lambda_R(M)$ denote the length of an $R$-module $M$. When there is no danger of confusion, we will
denote $\lambda_R(M)$ by $\lambda(M)$.

Let $F$ be  a finitely generated free $R$-module of rank $\gamma$, and let $E$ be a  submodule of $F$ of rank $e$. 
Let $S=R[F]=\mbox{Sym}(F)=\bigoplus_{k\ge 0}F^k$, and let $R[E]=\bigoplus_{k\ge0}E^k$ be the $R$-subalgebra of $S$ generated by $E$. Let
$$
E^k:_{F^k}\mathfrak m^{\infty} = \cup_{n=1}^{\infty}E^k:_{F^k}\mathfrak m^n
$$
denote the saturation of $E^k$ in $F^k$.

Let $F_L^k=F^k\otimes_R L$ (where $L$ is the quotient field of $R$), so that we have natural
embeddings $E^k\subset F^k\subset F_L^k$ for all $k$. Let $X=\mbox{Spec}(R)$, $\widetilde{E^k}$ be the sheafification of $E$ on $X$ and let $u_1,\ldots,u_s$ be generators of the ideal $\mathfrak m$.

There are identities
\begin{equation}\label{eqE'}
H^0(X\setminus\{\mathfrak m\},\widetilde{E^k})=\cap_{i=1}^s(E^k)_{u_i}=E^k:_{F_L^k}\mathfrak m^{\infty}.
\end{equation}
From the  exact sequence of cohomology groups
$$
0\rightarrow H^0_{\mathfrak m}(E^k)\rightarrow E^k\rightarrow H^0_{\mathfrak m}(X\setminus \{\mathfrak m\},\widetilde{E^k})\rightarrow H^1_{\mathfrak m}(E^k)\rightarrow 0,
$$
we deduce that we have isomorphisms of $R$-modules 
\begin{equation}\label{eqF'}
H^1_{\mathfrak m}(E^k)\cong E^k:_{F_L^k}\mathfrak m^{\infty}/E^k
\end{equation}
for $k\ge 0$. The same calculation for $F^k$ shows that 
\begin{equation}\label{eqQ'}
H^1_{\mathfrak m}(F^k)\cong F^k:_{F_L^k}\mathfrak m^{\infty}/F^k.
\end{equation}

From the left exact local cohomology sequence
$$
0\rightarrow H^0_{\mathfrak m}(F^k/E^k)\rightarrow H^1_{\mathfrak m}(E^k)\rightarrow H^1_{\mathfrak m}(F^k),
$$
we have that
\begin{equation}\label{eqI'}
H^0_{\mathfrak m}(F^k/E^k)\cong \left(E^k:_{F_L^k}\mathfrak m^{\infty})\cap F^k\right)/E^k=E^k:_{F^k}\mathfrak m^{\infty}/E^k.
\end{equation}

From (\ref{eqE'}), and the fact that $F^k$ is a free $R$-module, we have that $H^0(X\setminus \{\mathfrak m\},\widetilde{F^k})=F^k$ and
\begin{equation}\label{eqU}
E^k:_{F_L^k}\mathfrak m^{\infty}=E^k:_{F^k}\mathfrak m^{\infty}\mbox{ if $R$ has depth $\ge 2$.}
\end{equation}

Let $ES$  be the ideal of $S$ generated by $E$.
We compute the degree $n$ part of $(ES)^n$ from the formula
\begin{equation}\label{eq0}
\left[(ES)^n\right]_n=E^n.
\end{equation}

Let $R[\mathfrak mE]=\bigoplus_{n\ge 0}(\mathfrak m E)^n$ be the $R$-subalgebra of $S$ generated by $\mathfrak mE$.

Let $X=\mbox{Spec}(R)$, $Y=\mbox{Proj}(R[\mathfrak mE])$ and
$Z=\mbox{Proj}(R[E])$. 

Write $R[E]=R[\overline x_1,\ldots,\overline x_t]$ as a standard graded $R$-algebra, with 
$\mbox{deg }\overline x_i=1$ for all $i$. 
For $1\le i\le t$, let
$$
R_i=R[\frac{\overline x_1}{\overline x_i},\ldots,\frac{\overline x_t}{\overline x_i}],
$$
and let $V_i=\mbox{Spec}(R_i)$ for $1\le i\le t$.
$\{V_i\}$ is an affine cover of $Z$. Let $u_1,\ldots, u_s$ be generators of the ideal $\mathfrak m$. 
For $1\le i\le s$ and $1\le j\le t$, let
$$
R_{i,j}=R[\frac{u_{\alpha}\overline x_{\beta}}{u_i\overline x_j}\mid 1\le\alpha\le s, 1\le \beta\le t],
$$
and $U_{i,j}=\mbox{Spec}(R_{i,j})$.
Then $\{U_{i,j}\}$ is an affine cover of $Y$. Since
$$
R_j[\frac{u_1}{u_i},\ldots,\frac{u_s}{u_i}]=R_{i,j},
$$
we see that $Y$ is the blow up of the ideal sheaf $\mathfrak m\mathcal O_Z$.

The structure morphism $f:Y\rightarrow X$  factors as
a sequence of projective morphisms 
$$
Y\stackrel{g}{\rightarrow} Z\stackrel{h}{\rightarrow} X,
$$
where $Y$ is the blow up the ideal sheaf $\mathfrak m\mathcal O_Z$.
Define line bundles on $Y$ by $\mathcal L=g^*\mathcal O_Z(1)$ and $\mathcal M=\mathfrak m\mathcal O_Y$. Then $\mathcal O_Y(1)\cong \mathcal M\otimes \mathcal L$.  

We have $\mathcal O_Z(1)|V_j=\overline x_j\mathcal O_{V_j}$, $\mathcal L|U_{i,j}=\overline x_j\mathcal O_{U_{i,j}}$ and $\mathcal M|U_{i,j}=u_i\mathcal O_{U_{i,j}}$.

We give three consequences  (Proposition \ref{FG}, Proposition \ref{PropS2}
and Corollary \ref{CorS3}) of Serre's fundamental theorem for projective morphisms which will be useful.

\begin{Proposition}\label{FG} $\bigoplus_{k\ge 0}H^i(Y,\mathcal L^k)$ are finitely generated $R[E]$-modules for
all $i\in\NN$.
\end{Proposition}

\begin{proof} Let $\widetilde{E^k}$ be the sheafication of $E^k$ on $X$. From the natural surjections 
for $k\ge 0$ of
 $\mathcal O_Z$-modules $g^*(\widetilde{E^k})\rightarrow \mathcal O_Z(k)$, we obtain surjections $f^*(\widetilde{E^k})\rightarrow \mathcal L^k$ of $\mathcal O_Y$-modules, and a surjection
 $f^*(\bigoplus_{k\ge 0}\widetilde{E^k})\rightarrow \bigoplus_{k\ge 0}\mathcal L^k$. Hence $\bigoplus_{k\ge 0}\mathcal L^k$ is a finitely generated $f^*(\bigoplus_{k\ge 0}\widetilde{E^k})$-module. By Theorem III.2.4.1 \cite{EGA}, $R^if_*(\bigoplus_{k\ge 0}\mathcal L^k)$
 is a finitely generated $\bigoplus_{k\ge 0}\widetilde{E_k}$-module for $i\in\NN$. Taking global sections on the affine $X$, we obtain the conclusions of the proposition.
 \end{proof}
 
 \begin{Proposition}\label{PropS2} Suppose that $A$ is a Noetherian ring, and $B=\bigoplus_{k\ge 0}B_k$ is a finitely generated graded $A$-algebra, which is generated by $B_1$ as an $A$-algebra. Let $C=\mbox{Spec}(A)$ and $D=\mbox{Proj}(B)$. Let $\alpha:D\rightarrow C$ be the structure morphism.
 Then there exists a positive integer $\overline k$ such that $B_k=\Gamma(D,\mathcal O_D(k))$ for $k\ge \overline k$.
 \end{Proposition}
 
 \begin{proof} The ring $\bigoplus_{k\ge 0}\Gamma(D,\mathcal O_D(k))$ is a finitely generated graded
 $B$-module by Theorem III.2.4.1 \cite{EGA}. Hence $(\bigoplus_{k\ge 0}\Gamma(D,\mathcal O_D(k)))/B$
 is a finitely generated graded $B$-module. Since every element of this module is $B_+=\bigoplus_{k>0}B_k$ torsion, we have that $B_k/E_k=0$ for $k \gg0$.
 \end{proof}
 
 Taking the maximum over the $\overline k$ obtained from the above proposition applied to a finite affine cover of $W$, we obtain the following generalization of Proposition \ref{PropS2}.
 
 \begin{Corollary}\label{CorS3} Suppose that $W$ is a Noetherian scheme and $\mathcal B=\bigoplus_{k\ge 0}\mathcal B_k$ is a finitely generated graded $\mathcal O_W$-algebra, which is locally generated by $\mathcal B_1$ as a $\mathcal O_W$-algebra. Let $W'=\mbox{Proj}(\mathcal B)$ and let $\alpha:W'\rightarrow W$
be the structure morphism. Then there exists a positive integer $\overline k$ such that $\mathcal B_k=\alpha_*\mathcal O_{W'}(k)$ for $k\ge \overline k$,
\end{Corollary}

\section{Asymptotic Growth}

\begin{Proposition}\label{Prop6*}   Let $(R,\mathfrak m)$ be a  local domain of depth $\ge 2$. Let $d$ be the dimension of $R$. Suppose that $E$ is a rank $e$ $R$-submodule of a finitely generated free $R$-module $F$. Let  notation be as above.
Then there exist positive integers $k_0$, $k_1$  and $\tau$ such that 
\begin{enumerate}
\item[1)] for $k\ge k_0$, $n\in\ZZ$  and $\mathfrak p\in X\setminus\{\mathfrak m\}$,
$$
\Gamma(Y,\mathcal M^n\otimes\mathcal L^k)_{\mathfrak p}=(E^k)_{\mathfrak p}.
$$
\item[2)] For $k\ge k_1$,
$$
E^k:_{F^k}\mathfrak m^{\infty}=\Gamma(Y,\mathcal M^{-k\tau}\otimes\mathcal L^k).
$$
\end{enumerate}

\end{Proposition}

\begin{proof} We first establish 1). $U_i=\mbox{Spec}(R_{u_i})$ for $1\le i\le s$ is an affine cover of $X\setminus\{\mathfrak m\}$. $g|f^{-1}(U_i)$ is an isomorphism; in fact
$$
f^{-1}(U_i)=\mbox{Proj}(R[\mathfrak mE]_{u_i})=\mbox{Proj}(R[E]_{u_i})=h^{-1}(U_i).
$$
By Proposition \ref{PropS2}, there exist positive integers $a_i$ such that
$$
\Gamma(f^{-1}(U_i),\mathcal M^{-n}\otimes\mathcal L^k)=\Gamma(h^{-1}(U_i),\mathcal O_Z(k))
=(E^k)_{u_i}
$$
for $k\ge a_i$. Let $k_0=\mbox{max}\{a_1,\ldots,a_s\}$. Then for $\mathfrak p\in U_i$
$$
\Gamma(Y,\mathcal M^{-n}\otimes\mathcal L^k)_{\mathfrak p} =\Gamma(f^{-1}(U_i),\mathcal M^n\otimes\mathcal L^k)_{\mathfrak p}=(E^k)_{\mathfrak p}
$$
for $k\ge k_0$, establishing 1).

We now establish 2). Suppose that $n\ge 0$, and $k\ge 0$.   
Suppose that $\sigma\in E^k:_{F^k}\mathfrak m^n$. Let $i,j$ be such that
$1\le i\le s$ and $1\le j\le t$. 
$\sigma\mathfrak m^n\subset E^k$ implies $u_i^n\sigma\in E^k$ which implies there is an expansion
$$
u_i^n\sigma=\sum_{n_{1}+\cdots+n_{t}=k}r_{n_1,\ldots,n_t}\overline x_1^{n_{1}}\cdots\overline x_t^{n_{t}}
$$
with $r_{n_1,\ldots,n_t}\in R$. Thus 
$$
u_i^n\sigma=\overline x_j^k\left(\sum_{n_{1}+\cdots+n_{t}=k}r_{n_1,\ldots,n_t}(\frac{\overline x_1}{\overline x_j})^{n_{1}}\cdots(\frac{\overline x_t}{\overline x_j})^{n_{t}}\right),
$$
so that
$\sigma\in u_i^{-n}\overline x_j^kR_{i,j}$.
Thus 
$$
\sigma\in \cap_{i,j}u_i^{-n}\overline x_j^kR_{i,j}=\Gamma(Y,\mathcal M^{-n}\otimes\mathcal L^k).
$$
We have established that for $k\ge 0$ and $n\ge 0$,
\begin{equation}\label{eq2*}
E^k:_{F^k}\mathfrak m^n\subset \Gamma(Y,\mathcal M^{-n}\otimes\mathcal L^k).
\end{equation}

Recall that $S$ is a polynomial ring
$S=R[y_1,\ldots,y_{\gamma}]$ over $R$, where $\gamma$ is the rank of $F$. Let $W=\mbox{Proj}(S)$, with natural morphism $\alpha:W\rightarrow X$.  Let $\mathcal I$ be the sheafication of the graded ideal $ES$ on $W$.
We have expansions
$$
\overline x_i=\sum_{l=1}^{\gamma} f_{il}y_l
$$
with $f_{il}\in R$.

The inclusion $R[E]\subset S$ induces a rational map from $W$ to $Z$.

Let $\beta:W'\rightarrow W$ be the blow up of the ideal sheaf $\mathcal I$. Let $\mathcal N=\mathcal I\mathcal O_{W'}$ be the induced line bundle. $W'$ has an affine cover 
$A_{i,j}=\mbox{Spec}(T_{ij})$ for $1\le i\le  t$ and $1\le j\le \gamma$ with
$$
T_{ij}=R[\frac{y_1}{y_j},\ldots,\frac{y_{\gamma}}{y_j}][\frac{\overline x_1}{\overline x_i},\ldots,\frac{\overline x_t}{\overline x_i}].
$$
From the inclusions
$$
R_i=R[\frac{\overline x_1}{\overline x_i},\ldots,\frac{\overline x_t}{\overline x_i}]\subset T_{ij}
$$
we have  induced  morphisms $A_{i,j}\rightarrow V_i=\mbox{Spec}(R_i)$ which patch to give a morphism $\phi:W'\rightarrow Z$ which is a resolution of indeterminacy of the rational map
from $W$ to $Z$.

We calculate for all $i,j$,
$$
\phi^*(\mathcal O_Z(1))\mid A_{i,j}=\overline x_i\mathcal O_{A_{ij}}
=y_j(\sum_l f_{i,l}\frac{y_l}{y_j})\mathcal O_{A_{ij}}
=\left(\beta^*\mathcal O_W(1)\right)\mathcal I|A_{i,j},
$$
to see that 
$$
(\beta^*\mathcal O_W(1))\otimes\mathcal N\cong \phi^*\mathcal O_Z(1).
$$
By Corollary \ref{CorS3}, there exists a positive integer $k_1\ge k_0$ such that $\beta_*\mathcal N^k=\mathcal I^k$
for $k\ge k_1$. From the natural inclusion $\mathcal O_Z(k)\subset \phi_*\phi^*\mathcal O_Z(k)$, we have
by the projection formula that for $k\ge k_1$,
\begin{equation}\label{eq3*}
\begin{array}{lll}
h_*\mathcal O_Z(k)&\subset& h_*\phi_*(\phi^*\mathcal O_Z(k))
=\alpha_*\beta_*(\beta^*\mathcal O_W(k)\otimes\mathcal N^k)\\
&=&\alpha_*[\mathcal O_W(k)\otimes\beta_*\mathcal N^k]=\alpha_*[\mathcal O_W(k)\otimes\mathcal I^k]\\
&\subset&\alpha_*\mathcal O_W(k)=\widetilde{F^k},
\end{array}
\end{equation}
where $\widetilde{F^k}$ is the sheafication of the $R$-module $F$ on $X$.
Now we have
\begin{equation}\label{eq4*}
\begin{array}{lll}
\Gamma(Y,\mathcal M^{-n}\otimes\mathcal L^k)&=&
\Gamma(X,f_*(\mathcal M^{-n}\otimes\mathcal L^k))\\
&\subset& \Gamma(X\setminus\{\mathfrak m\},f_*(\mathcal M^{-n}\otimes\mathcal L^k))
=\Gamma(X\setminus\{\mathfrak m\},h_*\mathcal O_Z(k))\\
&\subset&\Gamma(X\setminus\{\mathfrak m\},\widetilde{F^k})=F^k
\end{array}
\end{equation}
since $R$, and hence the free $R$-module $F^k$, have depth $\ge 2$.

From (\ref{eq0}), we deduce that for $k,n\ge 0$,
\begin{equation}\label{eq5*}
\left((ES)^k:_S\mathfrak m^nS\right)\cap F^k=E^k:_{F^k}\mathfrak m^n.
\end{equation}

By  1.5 \cite{KM} or Theorem 1.3 \cite{S}, there exists a positive integer $\tau$ such that
$$
(ES)^k:_S\mathfrak m^{k\tau}S=(ES)^k:_S(\mathfrak mS)^{\infty}
$$
for all $k\ge 0$. Thus from (\ref{eq5*}) we have that
\begin{equation}\label{eq6*}
E^k:_{F^k}\mathfrak m^{k\tau}=E^k:_{F^k}\mathfrak m^{\infty}
\end{equation}
for $k\ge 0$. From (\ref{eq6*}), (\ref{eq2*}) and (\ref{eq4*}), we have inclusions
$$
E^k:_{F^k}\mathfrak m^{\infty}\subset \Gamma(Y,\mathcal M^{-k\tau}\otimes\mathcal L^k)\subset F^k
$$
for $k\ge k_1$. The conclusions of 2) of the proposition now follow from 1) of the proposition  since
$E^k:_{F^k}\mathfrak m^{\infty}$ is the largest $R$-submodule $N$ of $F^k$ which has the property that
$N_{\mathfrak p}=(E^k)_{\mathfrak p}$ for $\mathfrak p\in X\setminus\{\mathfrak m\}$.

\end{proof}

\begin{Theorem}\label{Theorem1*} Suppose that $(R,\mathfrak m)$ is a  local domain of depth $\ge 2$ which is essentially of finite type over a field $K$ of characteristic zero. Let $d$ be the dimension of $R$. 
Suppose that $E$ is a rank $e$ submodule of a finitely generated free $R$-module $F$.
Then the limit
$$
\lim_{k\rightarrow \infty}\frac{\lambda\left(E^k:_{F^k}\mathfrak m^{\infty}/E^k\right)}{k^{d+e-1}}\in \RR
$$
exists.
\end{Theorem}

\begin{proof}
Let notation be as above.

First consider  
the short exact sequences
\begin{equation}\label{eq11*}
0\rightarrow \Gamma(Y,\mathcal L^k)/E^k\rightarrow E^k:_{F^k}\mathfrak m^{\infty}/E^k\rightarrow
  E^k:_{F^k}\mathfrak m^{\infty}/\Gamma(Y,\mathcal L^k)\rightarrow 0.
 \end{equation}

 $\bigoplus_{k\ge 0}\Gamma(Y,\mathcal L^k)$
 is a finitely generated $R[E]$-module by Lemma \ref{FG}.
 By 1) of Proposition \ref{Prop6*},  the support of the $R$-module 
 $\Gamma(Y,\mathcal L^k)/E^k$ is contained in $\{\mathfrak m\}$ for all $k$. Since 
 $(\bigoplus_{k\ge 0}\Gamma(Y,\mathcal L^k))/R[E]$ is a finitely generated $R[E]$-module,
 there is a positive integer $r$ such that $\mathfrak m^r(\Gamma(Y,\mathcal L^k)/E^k)=0$ for all $k$.
 Since $\mbox{dim }R[E]/\mathfrak mR[E]\le \mbox{dim }R+\mbox{rank }E-1=d+e-1$, 
 and $R/\mathfrak m^r$ is an Artin local ring, we have by the Hilbert-Serre theorem that  $\lambda(\Gamma(Y,\mathcal L^k)/E^k)$ is a polynomial of degree less than or equal to $d+e-2$ for $k\gg 0$.
 Thus there exists a constant $\alpha$ such that $\lambda(\Gamma(Y,\mathcal L^k)/E^k)\le\alpha k^{d+e-2}$ for all $k$.
 From (\ref{eq11*}), we are now reduced to showing that the limit
 $$
 \lim_{k\rightarrow\infty}
 \frac{\lambda( E^k:_{F^k}\mathfrak m^{\infty}/\Gamma(Y,\mathcal L^k))}{k^{d+e-1}}
 $$
 exists, from which we will have 
 \begin{equation}\label{eq12*}
 \lim_{k\rightarrow\infty}
 \frac{\lambda( E^k:_{F^k}\mathfrak m^{\infty}/\Gamma(Y,\mathcal L^k))}{k^{d+e-1}}=
\lim_{k\rightarrow\infty}\frac{
\lambda( E^k:_{F^k}\mathfrak m^{\infty}/E^k)}{k^{d+e-1}}.
\end{equation}

Taking global sections of the short exact sequences
$$
0\rightarrow \mathcal L^k\rightarrow \mathcal M^{-k\tau}\otimes \mathcal L^k\rightarrow
\mathcal M^{-k\tau}\otimes\mathcal L^k\otimes(\mathcal O_Y/\mathfrak m^{k\tau}\mathcal O_Y)\rightarrow 0,
$$
we obtain by Proposition \ref{Prop6*} left exact sequences
\begin{equation}\label{eq7*}
0\rightarrow E^k:_{F^k}\mathfrak m^{\infty}/\Gamma(Y,\mathcal L^k)\rightarrow \Gamma(Y,\mathcal M^{-k\tau}\otimes \mathcal L^k
\otimes (\mathcal O_Y/\mathfrak m^{k\tau}\mathcal O_Y))\rightarrow H^1(Y,\mathcal L^k)
\end{equation}
for $k\ge k_1$.

Let $u_1,\ldots,u_s$ be generators of the ideal $\mathfrak m$, and set $U_i=\mbox{Spec}(R_{u_i})$, so that $\{U_1,\ldots,U_s\}$ is
an affine cover of $X\setminus\{\mathfrak m\}$. Then $\mathcal L|f^{-1}(U_i)$ is ample, so there exist
 positive integers $b_i$ such that $R^1f_*(\mathcal L^k)\mid U_i=0$ for $k\ge b_i$. Let $k_2=\mbox{max}\{b_1,\ldots,b_s\}$. We have that the support of  $H^1(Y,\mathcal L^k)$ is contained in $\{\mathfrak m\}$ for $k\ge k_2$.

$\bigoplus_{k\ge 0}H^1(Y,\mathcal L^k)$ is a finitely generated $R[E]$-module by Lemma \ref{FG}. Hence the submodule
$M=\bigoplus_{k\ge k_2}H^1(Y,\mathcal L^k)$ is a finitely generated graded $R[E]$-module.  We have that $\mathfrak m^rM=0$
for some positive integer $r$. Since
$$
\mbox{dim }R[E]/\mathfrak mR[E]\le\mbox{dim }R+\mbox{rank }E-1=d+e-1,
$$
and $R/\mathfrak m^r$ is an Artin local ring, we have by the Hilbert-Serre theorem that  $\lambda(H^1(Y,\mathcal L^k))$ is a polynomial of degree less than or equal to $d+e-2$ for $k\gg 0$. Thus
 there exists a constant $c$ such that 
$$
\lambda(H^1(Y,\mathcal L^k))\le ck^{d+e-2}
$$
for all $k\ge 0$. By consideration of (\ref{eq12*}) and (\ref{eq7*}), we are reduced to proving that the limit
\begin{equation}\label{eq9*}
\lim_{k\rightarrow \infty}\frac{\lambda(H^0(Y,\mathcal M^{-k\tau}\otimes\mathcal L^k\otimes\mathcal O_Y/\mathfrak m^{k\tau}\mathcal O_Y))}{k^{d+e-1}}
\end{equation}
exists.

If $R/\mathfrak m$ is algebraic over $K$, let $K'=K$. If $R/\mathfrak m$ is transcendental over $K$, let $t_1,\ldots,t_r$ be a lift of a transcendence basis of $R/\mathfrak m$ over $K$ to $R$. The rational function field $K(t_1,\ldots,t_r)$ is contained in $R$. Let $K'=K(t_1,\ldots,t_r)$. We have that $R/\mathfrak m$ is finite algebraic over $K'$. 

There exists an  affine $K'$-variety $X'=\mbox{Spec}(A)$ such that $R$ is the local ring of a
closed point $\alpha$ of $X'$, and $E$ extends to a submodule $E'$ of $A^{\gamma}$, where $\gamma$ is the rank of the free $R$-module $F$. We then have an inclusion of
graded $A$-algebras $A[E']\subset \mbox{Sym}(A^{\gamma})$ which extends $R[E]$. 
Identify $\mathfrak m$ with its extension to a maximal ideal of $A$.
The structure morphism
$Y'=\mbox{Proj}(A[\mathfrak m E'])\rightarrow X'$ is projective and its localization at  $\mathfrak m$
 is  $f:Y\rightarrow X$.
Let $\overline X$ be a projective closure of $X'$ and let $\tilde Y$ be a projective closure of $Y'$.
$X'$ is an open subset of $\overline X$ and $Y'$ is an open subset of $\tilde Y$. 
Let $\overline Y\rightarrow \tilde Y$ be the blow up of an ideal sheaf which gives a resolution of indeterminancy of the rational map from $\tilde Y$ to $\overline X$. We may assume that the morphism
$\overline Y\rightarrow \tilde Y$ is an isomorphism over the locus where the rational map is a morphism, and thus an isomorphism over the  subset $Y'$ of $\tilde Y$. Let $\overline f:\overline Y\rightarrow \overline X$ be the resulting morphism. We now establish that $\overline f^{-1}(X')=Y'$. Suppose that $p\in X'$ and $q\in \overline f^{-1}(p)$. Let $V$ be a valuation ring of the function field $L$ of $\overline Y$ (which is also the function field of $Y'$) which dominates the local ring $\mathcal O_{\overline Y,q}$.
By assumption, $V$ dominates the local ring $\mathcal O_{X',p}$. $V$ dominates the local ring of a point on $Y'$, by the valuative criterion for properness (Theorem II.4.7 \cite{H})
applied to the proper morphism $Y'\rightarrow X'$. Since
$V$ dominates the local ring of a unique point on $\overline Y$, we have that 
$q\in Y'$.

After possibly replacing $\overline Y$ with the blow up of an ideal sheaf on $\overline Y$ whose support is
disjoint from $Y'$, we may assume that $\mathcal L$ extends to a line bundle on $\overline Y$ which we will also denote by $\mathcal L$. We will identify $\mathfrak m$ with its extension to the ideal sheaf of the point
$\alpha$ on $\overline X$, and identify $\mathcal M$ with its extension $\mathfrak m\mathcal O_{\overline Y}$
to a line bundle on $\overline Y$.
Let $\mathcal A$ be an ample divisor on $\overline X$. Then there
exists $l>0$ such that $\mathcal C=\overline f^*(\mathcal A^l)\otimes\mathcal L$
is generated by global sections and is big.

Set $\mathcal B=\mathcal C\otimes \mathcal M^{-\tau}$. Tensor the short exact sequences
$$
0\rightarrow \mathcal M^{k\tau}\rightarrow \mathcal O_{\overline Y}\rightarrow
\mathcal O_{\overline Y}/\mathfrak m^{k\tau}\mathcal O_{\overline Y}\cong \mathcal O_Y/\mathfrak m^{k\tau}\mathcal O_Y\rightarrow 0
$$
with $\mathcal B^k$ to obtain the short exact sequences
$$
0\rightarrow \mathcal C^k\rightarrow \mathcal B^k\rightarrow \mathcal M^{-k\tau}\otimes \mathcal L^k\otimes
\mathcal O_Y/\,\mathfrak m^{k\tau}\mathcal O_Y\rightarrow 0
$$
for $k\ge 0$. Taking global sections, we have exact sequences
\begin{equation}\label{eq10*}
0\rightarrow H^0(\overline Y,\mathcal C^k)\rightarrow H^0(\overline Y,\mathcal B^k)\rightarrow 
H^0(Y,\mathcal M^{-k\tau}\otimes \mathcal L^k\otimes
\mathcal O_Y/\,\mathfrak m^{k\tau}\mathcal O_Y)\rightarrow H^1(\overline Y, \mathcal C^k).
\end{equation}

For a coherent sheaf $\mathcal F$ on $\overline Y$, let
$$
h^i(\overline Y,\mathcal F)=\mbox{dim}_{K'}H^i(\overline Y,\mathcal F).
$$
Since $\mathcal C$ is semiample (generated by global sections and big) and $\overline Y$ has dimension $d+e-1$, we have that 
$$
\lim_{k\rightarrow \infty}\frac{h^1(\overline Y,\mathcal C^k)}{k^{d+e-1}}=0.
$$
This follows for instance from \cite{F1}.
Since $\bigoplus_{k\ge 0}H^0(\overline Y,\mathcal C^k)$ is a finitely generated $K'$ algebra of dimension 
$d+e$, as $\mathcal C$ is generated by global sections and is big (or by the Riemann Roch theorem and the vanishing theorem of \cite{F1}) we have that
the limit
$$
\lim_{k\rightarrow\infty}\frac{h^0(\overline Y,\mathcal C^k)}{k^{d+e-1}}\in\QQ
$$
exists. Since $\mathcal B$ is big, by the corollary to \cite{F2} given in Example 11.4.7 \cite{L} or \cite{CHST},
we have that 
the limit
$$
\lim_{k\rightarrow\infty}\frac{h^0(\overline Y,\mathcal B^k)}{k^{d+e-1}}\in\RR
$$
exists. From the sequence (\ref{eq10*}), we see that
$$
\lim_{k\rightarrow \infty}\frac
{h^0(Y,\mathcal M^{-k\tau}\otimes \mathcal L^k\otimes
\mathcal O_Y/\,\mathfrak m^{k\tau}\mathcal O_Y)}{k^{d+e-1}}\in \RR
$$
exists. The conclusions of the theorem now follow from (\ref{eq9*}) and the formula
$$
\begin{array}{lll}
h^0(Y,\mathcal M^{-k\tau}\otimes \mathcal L^k\otimes
\mathcal O_Y/\,\mathfrak m^{k\tau}\mathcal O_Y)&=&\mbox{dim}_{K'}
H^0(Y,\mathcal M^{-k\tau}\otimes \mathcal L^k\otimes
\mathcal O_Y/\,\mathfrak m^{k\tau}\mathcal O_Y)\\
&=&[R/\mathfrak m:K']\lambda(H^0(Y,\mathcal M^{-k\tau}\otimes \mathcal L^k\otimes
\mathcal O_Y/\,\mathfrak m^{k\tau}\mathcal O_Y)).
\end{array}
$$
\end{proof}

\begin{Remark}\label{charp} The conclusions of Theorem \ref{Theorem1*} are also true if $K$ is a perfect field of
positive characteristic and $R/\mathfrak m$ is algebraic over $K$. In this case we have that $K'=K$ in the proof of Theorem \ref{Theorem1*}. Let $\overline K$ be an algebraic closure of $K$. Since $K$ is perfect, $\overline Y\times_K\overline K$ is reduced, and to compute the limit, we reduce to computing the sections of the pullback of $\mathcal B^k$ on the disjoint union of the irreducible (integral) components of $\overline Y\times_K\overline K$.  Fujita's approximation theorem is valid on varieties over an algebraically closed field of positive characteristic, as was shown by Takagi \cite{Ta}, from which the existence of the limit now follows.
\end{Remark}

\begin{Remark} Theorem \ref{Theorem1*} is proven for graded ideals in \cite{CHST}. An example where the limit is an irrational number is given in \cite{CHST}. 
The  theorem is proven with the additional assumptions that $R$ is regular, $E=I$
is an ideal in $F=R$, and the singular locus of $\mbox{Spec}(R/I)$ is $\mathfrak m$ in \cite{CHS}. 
Kleiman \cite{K} has proven Theorem \ref{Theorem1*} in the case that $E$ is a direct summand of $F$ locally at every nonmaximal prime of $R$.
\end{Remark}

\begin{Corollary}\label{LC} Suppose that $(R,\mathfrak m)$ is a  local domain of depth $\ge 2$ 
which is essentially of finite type over a field $K$ of characteristic zero. Let $d$ be the dimension of $R$. Suppose that $E$ is a rank $e$ submodule of a finitely generated free $R$-module $F$. Then the limit
$$
\lim_{k\rightarrow \infty}\frac{(d+e-1)!}{k^{d+e-1}}\lambda(H^0_{\mathfrak m}(F^k/E^k))\in\RR
$$
exists. 
Thus the epsilon multiplicity $\epsilon(E)$ of the module $E$, defined in  \cite{UV} as a limsup, actually exists as a limit.
\end{Corollary}
The example of \cite{CHST} shows that $\epsilon(E)$ may be an irrational number.

\begin{proof} The corollary is immediate from Theorem \ref{Theorem1*} and (\ref{eqI'}).
\end{proof}

\begin{Remark}\label{charp2} The conclusions of Corollary \ref{LC} are valid if $K$ is a perfect field of positive characteristic  and $R/\mathfrak m$ is algebraic over $K$, by Remark \ref{charp}.
\end{Remark}

\section{Extension to domains of dimension $\ge 2$.}\label{Exts}

In this section, we prove  extensions of Theorem \ref{ThmInt} and Corollary \ref{CorInt} to domains of dimension $\ge 2$. Let notation be as in Section \ref{Prep}.

Suppose  that $R$ is a domain of dimension $d\ge 2$ with a dualizing module. By the Theorem of Finiteness, Theorem VIII.2.1 (and footnote) \cite{SGAII}, 
\begin{equation}\label{eqG}
\overline R=
\Gamma(X\setminus \{\mathfrak m\},\mathcal O_X)=\cap _{\mathfrak p\in X\setminus \{\mathfrak m\}}R_{\mathfrak p}
\end{equation}
is a finitely generated $R$-module, which lies between $R$ and its quotient field. Since $\overline R/R$ is $\mathfrak m$-torsion, 
\begin{equation}\label{eqS}
\lambda_R(\overline R/R)<\infty.
\end{equation}
Let $\mathfrak m_1,\ldots, \mathfrak m_\alpha$ be the maximal ideals of $\overline R$ which lie over $\mathfrak m$.  By our construction, 
$$
0=H^1_{\mathfrak m}(\overline R)=H^1_{\mathfrak m\overline R}(\overline R)=\bigoplus_{i=1}^{\alpha}H^1_{\mathfrak m_i\overline R}(\overline R),
$$
so 
$$
H^1_{\mathfrak m_i\overline R_{\mathfrak m_i}}(\overline R_{\mathfrak m_i})=H^1_{\mathfrak m_i\overline R}(\overline R)\otimes_{\overline R}\overline R_{\mathfrak m_i}=
0
$$
 for $1\le i \le \alpha$, and thus
$\mbox{depth}(\overline R_{\mathfrak m_i})\ge 2$ for $1\le i\le \alpha$.

Let $\overline F=F\otimes_R\overline R$ and $\overline R[\overline F]=\bigoplus_{k\ge 0}\overline F^k$,
so that $\overline F^k\cong F^k\otimes_R\overline R$ for all $k$. Let $\overline E=\overline RE$ be the  $\overline R$-submodule of $\overline F$ generated by $E$. Let $\overline R[\overline E]=\bigoplus_{k\ge 0}\overline E^k$ be the $\overline R$-subalgebra of $\overline R[\overline F]$ generated by $\overline E$.

Let $u_1,\ldots,u_s$ be
 generators of the ideal $\mathfrak m$. For $k\in\NN$, 
let $\widetilde{E^k}$ be the sheafification of $E^k$ on $X=\mbox{Spec}(R)$.

There are identities
\begin{equation}\label{eqE}
H^0(X\setminus\{\mathfrak m\},\widetilde{E^k})=\cap_{i=1}^s(E^k)_{u_i}
=E^k:_{\overline F^k}\mathfrak m^{\infty}.
\end{equation}
From the  exact sequence of cohomology groups
$$
0\rightarrow H^0_{\mathfrak m}(E^k)\rightarrow E^k\rightarrow H^0_{\mathfrak m}(X\setminus \{\mathfrak m\},\widetilde{E^k})\rightarrow H^1_{\mathfrak m}(E^k)\rightarrow 0,
$$
we deduce that we have isomorphisms of $R$-modules 
\begin{equation}\label{eqF}
H^1_{\mathfrak m}(E^k)\cong E^k:_{\overline F^k}\mathfrak m^{\infty}/E^k
\end{equation}
for $k\ge 0$. The same calculation for $F^k$ shows that 
\begin{equation}\label{eqQ}
H^1_{\mathfrak m}(F^k)\cong F^k:_{\overline F^k}\mathfrak m^{\infty}/F^k.
\end{equation}

From the left exact local cohomology sequence
$$
0\rightarrow H^0_{\mathfrak m}(F^k/E^k)\rightarrow H^1_{\mathfrak m}(E^k)\rightarrow H^1_{\mathfrak m}(F^k),
$$
we have that
\begin{equation}\label{eqI}
H^0_{\mathfrak m}(F^k/E^k)\cong \left(E^k:_{\overline F^k}\mathfrak m^{\infty})\cap F^k\right)/E^k=E^k:_{F^k}\mathfrak m^{\infty}/E^k.
\end{equation}

\begin{Theorem}\label{Theorem3*} Suppose that $(R,\mathfrak m)$ is a  local domain of
  dimension $d\ge 2$
 which is essentially of finite type over a field $K$ of characteristic zero
 (or over a perfect field $K$ such that $R/\mathfrak m$ is algebraic over $K$).
Suppose that $E$ is a rank $e$ submodule of a finitely generated free $R$-module $F$.
Then the limit
\begin{equation}\label{*eq1}
\lim_{k\rightarrow \infty}\frac{\lambda\left(E^k:_{\overline F^k}\mathfrak m^{\infty}/E^k\right)}{k^{d+e-1}}\in\RR
\end{equation}
exists.
\end{Theorem}

\begin{proof} Since
$\overline E^k:_{\overline F^k}\mathfrak m^{\infty}/\overline E^k$ are finitely generated $\mathfrak m\overline R$-torsion $\overline R$-modules, we have that
$$
\overline E^k:_{\overline F^k}\mathfrak m^{\infty}/\overline E^k\cong \bigoplus_{i=1}^{\alpha}\left(
\overline E_{\mathfrak m_i}^k:_{\overline F^k_{\mathfrak m_i}}\mathfrak m_i^{\infty}/\overline E^k_{\mathfrak m_i}\right).
$$

By Theorem \ref{ThmInt}, we have that
$$
\lim_{k\rightarrow \infty}
\frac{\lambda_{\overline R_{\mathfrak m_i}}\left(\overline E_{\mathfrak m_i}^k:_{\overline F_{\mathfrak m_i}^k}\mathfrak m_i^{\infty}/\overline E_{\mathfrak m_i}^k\right)}{k^{d+e-1}}
$$
exists for $1\le i\le \alpha$. Since for any $\overline R_{\mathfrak m_i}$ module $M$ we have that
$$
\lambda_R(M)=[\overline R/\mathfrak m_i:R/\mathfrak m]\lambda_{\overline R_{\mathfrak m_i}}(M),
$$
 we conclude that

\begin{equation}\label{eqK}
\lim_{k\rightarrow \infty}
\frac{\lambda_R(\overline E^k:_{\overline F^k}\mathfrak m^{\infty}/\overline E^k)}{k^{d+e-1}}
\end{equation}
exists.
We have
$$
\overline E^k:_{\overline F^k}\mathfrak m^{\infty}=\cap_{i=1}^s(\overline E^k)_{u_i}=\cap_{i=1}^s(E^k)_{u_i}
=E^k:_{\overline F^k}\mathfrak m^{\infty}.
$$
Consider the short exact sequences
\begin{equation}\label{eqL}
0\rightarrow \overline E^k/E^k\rightarrow E^k:_{\overline F^k}\mathfrak m^{\infty}/E^k\rightarrow
\overline E^k:_{\overline F^k}\mathfrak m^{\infty}/\overline E^k\rightarrow 0.
\end{equation}
Now $\overline R[\overline E]/R[E]$ is a finitely generated $R[E]$-module, and the support of the $R$-module
$\overline E^k/E^k$ is contained in $\{\mathfrak m\}$ for all $k$, so there exists a positive integer $r$ such that $\mathfrak m^r$ annihilates $\overline R[\overline E]/R[E]$. Thus (as in the argument following
equation (\ref{eq11*}) in the proof of Theorem \ref{Theorem1*}), we have that there exists a constant $\beta$ such that
\begin{equation}\label{eqT}
\lambda_R(\overline E^k/E^k)\le\beta k^{d+e-2}
\end{equation}
for all $k$. The conclusions of the proposition now follow from (\ref{eqK}), (\ref{eqT}) and (\ref{eqL}).
\end{proof}

I thank Craig Huneke, Bernd Ulrich and Javid Validashti for pointing out the following consequence of Theorem \ref{Theorem3*}.

\begin{Corollary}\label{Cor4*} Suppose that $(R,\mathfrak m)$ is a  local domain of
  dimension $d\ge 2$
 which is essentially of finite type over a field $K$ of characteristic zero
 (or over a perfect field $K$ such that $R/\mathfrak m$ is algebraic over $K$).
Suppose that $E$ is a rank $e$ submodule of a finitely generated free $R$-module $F$.
Suppose that $\gamma=\mbox{\rm rank}(F)<d+e$.
Then the limits
\begin{equation}\label{*eq2}
\lim_{k\rightarrow \infty}\frac{\lambda\left(E^k:_{F^k}\mathfrak m^{\infty}/E^k\right)}{k^{d+e-1}}\in \RR
\end{equation}
and
\begin{equation}\label{*eq3}
\lim_{k\rightarrow \infty}\frac{(d+e-1)!}{k^{d+e-1}}\lambda(H^0_{\mathfrak m}(F^k/E^k))\in \RR
\end{equation}
exist. In particular,  the epsilon multiplicity $\epsilon(E)$ of $E$ exists as a limit.
\end{Corollary}

\begin{proof} 
We will establish that the limit (\ref{*eq2}) exists.
We have exact sequences
\begin{equation}\label{eqA}
0\rightarrow 
E^k:_{F^k}\mathfrak m^{\infty}/E^k\rightarrow E^k:_{\overline F^k}\mathfrak m^{\infty}/E^k
\rightarrow 
E^k:_{\overline F^k}\mathfrak m^{\infty}/E^k:_{F^k}\mathfrak m^{\infty}\rightarrow 0
\end{equation}
and inclusions
\begin{equation}\label{eqB}
E^k:_{\overline F^k}\mathfrak m^{\infty}/E^k:_{F^k}\mathfrak m^{\infty}=E^k:_{\overline F^k}\mathfrak m^{\infty}/\left((E^k:_{\overline F^k}\mathfrak m^{\infty})\cap F^k\right)
\rightarrow F^k:_{\overline F^k}\mathfrak m^{\infty}/F^k
\end{equation}
for $k\ge 0$.

We  have
\begin{equation}\label{eqP}
F^k:_{\overline F^k}\mathfrak m^{\infty}/F^k=\overline F^k/F^k\cong (\overline R/R)^{\binom{k+\gamma-1}{\gamma-1}}.
\end{equation}
Since $\gamma=\mbox{rank}(F)<d+e$, we have
$$
\lim_{k\rightarrow\infty}\frac{\lambda_R \left(E^k:_{\overline F^k}\mathfrak m^{\infty}/E^k:_{F^k}\mathfrak m^{\infty} \right)}{k^{d+e-1}}=0.
$$
The existence of the limit (\ref{*eq2})  now follows from  (\ref{eqA}) and  Theorem \ref{Theorem3*}.
The existence of the limit (\ref{*eq3}) is  immediate from (\ref{*eq2})
and (\ref{eqI}).
\end{proof}

\end{document}